\documentclass[11pt,reqno]{amsart}
\usepackage{amssymb,mathrsfs,graphicx}
\usepackage{epsfig,subfigure}
\usepackage{ifthen}
\usepackage{colortbl}
\definecolor{black}{rgb}{0.0, 0.0, 0.0}
\definecolor{red}{rgb}{1.0, 0.5, 0.5}

\usepackage{ifthen} 
\provideboolean{shownotes} 
\setboolean{shownotes}{true} 
\newcommand{\margnote}[1]{
\ifthenelse{\boolean{shownotes}}%
{\marginpar{\raggedright\tiny\texttt{#1}}}%
{}%
}
\newcommand{\hole}[1]{
\ifthenelse{\boolean{shownotes}}%
{\begin{center} \fbox{ \rule {.25cm}{0cm} \rule[-.1cm]{0cm}{.4cm}
\parbox{.85\textwidth}{\begin{center} \texttt{#1}\end{center}} \rule
{.25cm}{0cm}}\end{center}} {} }

\title[ ]{Contractivity of the Wasserstein metric for the kinetic Kuramoto equation}

\author[Carrillo]{Jos\'{e} A. Carrillo}
\address[Jos\'{e} A. Carrillo]{\newline Department of Mathematics
    \newline Imperial College London, London SW7 2AZ, United Kingdom}
\email{carrillo@imperial.ac.uk}

\author[Choi]{Young-Pil Choi}
\address[Young-Pil Choi]{\newline Department of Mathematics
    \newline Imperial College London, London SW7 2AZ, United Kingdom}
\email{young-pil.choi@imperial.ac.uk}

\author[Ha]{Seung-Yeal Ha}
\address[Seung-Yeal Ha]{\newline Department of Mathematical Sciences and Research Institute of Mathematics
    \newline Seoul National University, Seoul 151-747, Korea}
\email{syha@snu.ac.kr}

\author[Kang]{Moon-Jin Kang}
\address[Moon-Jin Kang]{\newline Department of Mathematical Sciences
    \newline Seoul National University, Seoul 151-747, Korea}
\email{hiofte@snu.ac.kr}

\author[Kim]{Yongduck Kim}
\address[Yongduck Kim]{\newline Department of Mathematical Sciences
    \newline Seoul National University, Seoul 151-747, Korea}
\email{dikky5@snu.ac.kr}

\numberwithin{equation}{section}

\newtheorem{theorem}{Theorem}[section]
\newtheorem{lemma}{Lemma}[section]

\newtheorem{proposition}{Proposition}[section]
\newtheorem{remark}{Remark}[section]
\newtheorem{definition}{Definition}[section]

\newcommand{\bbr}{\mathbb R}
\newcommand{\bbs}{\mathbb S}
\newcommand{\bbt}{\mathbb T}
\newcommand{\bbp} {\mathbb P}
\newcommand{\bbz} {\mathbb Z}

\def\charf {\mbox{{\text 1}\kern-.30em {\text l}}}

\begin{document}
\allowdisplaybreaks

\date{\today}

\subjclass{92D25,74A25,76N10}

\keywords{Kuramoto model, phase, complete
synchronization, Wasserstein distance, contraction}

\thanks{\textbf{Acknowledgments.}
JAC acknowledges partial support by MINECO MTM2011-27739-C04-02,
by GRC 2009 SGR 345 from Generalitat de Catalunya, and by the
Engineering and Physical Sciences Research Council grant number
EP/K008404/1. JAC also acknowledges support from the Royal Society
through a Wolfson Research Merit Award. YPC was supported by Basic
Science Research Program through the National Research Foundation
of Korea (NRF) funded by the Ministry of Education, Science and
Technology (2012R1A6A3A03039496). The work of SYHA is supported by
NRF grant (2011-0015388).}

\begin{abstract}
We present synchronization and contractivity estimates for the
kinetic Kuramoto model obtained from the Kuramoto phase model in
the mean-field limit. For identical Kuramoto oscillators, we
present an admissible class of initial data leading to
time-asymptotic complete synchronization, that is, all measure
valued solutions converge to the traveling Dirac measure
concentrated on the initial averaged phase. If two initial Radon
measures have the same natural frequency density function and strength of coupling, we show
that the Wasserstein $p$-distance between corresponding measure
valued solutions is exponentially decreasing in time. This
contraction principle is more general than previous
$L^1$-contraction properties of the Kuramoto phase model.
\end{abstract}

\maketitle \centerline{\date}


%
%
\section{Introduction}
\setcounter{equation}{0} The objective of this paper is to present
a contraction property of the kinetic Kuramoto equation in the
Wasserstein metric. The synchronization phenomena exhibited by
various biological systems are ubiquitous in nature, e.g., the
flashing of fireflies, chorusing of crickets, synchronous firing
of cardiac pacemakers, and metabolic synchrony in yeast cell
suspensions (see for instance \cite{A-B, Bu}). Winfree and
Kuramoto \cite{Ku2, Wi2} pioneered the mathematical treatment of
these synchronized phenomena. They introduced phase models for
large weakly coupled oscillator systems, and showed that the
synchronized behavior of complex biological systems can emerge
from the competing mechanisms of intrinsic randomness and
sinusoidal couplings. The kinetic Kuramoto equation has been
widely used in the literature \cite{A-B} to analyze the phase
transition from a completely disordered state to a partially
ordered state as the coupling strength increases from zero.
Suppose that $g=g(\Omega)$ is an integrable steady probability
density function for natural frequencies with a compact support
(see \eqref{dist} for details). Let $f = f(\theta, \Omega, t)$ be
the probability density function of Kuramoto oscillators in
$\theta \in \bbt := \bbr/(2\pi \bbz)$ with a natural frequency
$\Omega$ at time $t$ as in \cite{La}. The kinetic Kuramoto
equation(KKE) is given as follows:
\begin{align}
\begin{aligned} \label{k-ku}
\displaystyle &\partial_t f + \partial_{\theta} (\omega[f] f) = 0, \qquad (\theta, \Omega) \in \bbt \times \bbr,~~t > 0, \\
& \omega[f](\theta, \Omega, t) = \Omega - K \int_{\bbt} \sin(\theta-\theta_*) \rho(\theta_*, t) d\theta_*, \quad \rho(\theta_*, t) := \int_{\bbr} f d\Omega_*,
\end{aligned}
\end{align}
subject to the initial data:
\begin{equation} \label{ini}
 f(\theta, \Omega, 0) = f_0(\theta, \Omega), \quad  \int_{\bbt} f_0 d\theta = g(\Omega).
\end{equation}

Note that KKE \eqref{k-ku} can be regarded as a scalar
conservation law with a nonlocal flux, and it has been derived
from a previously proposed Kuramoto model  \cite{Ch, La}. However,
to the best of authors' knowledge, few studies have investigated
the qualitative properties of the KKE, such as an asymptotic
behavior and stability of some equilibria. \newline

The main results of this paper can be summarized as follows.
First, we present sufficient conditions for the emergence of
completely synchronized states. More precisely, when many coupled
limit-cycle oscillators have the same natural frequency (identical
oscillators), and the support of the initial Radon measure is
confined in a half circle, we show that the measure-valued
solution approaches a multiple of the Dirac delta measure
concentrated on the initial average phase value and the common
natural frequency value time-asymptotically; in other words,
complete phase synchronization occurs asymptotically. For this
purpose, we lift the finite-dimensional result for the Kuramoto
model (KM) to the infinite-dimensional KKE. Second, we present a
contraction property of the KKE in the Wasserstein $p$-distance
for measure valued solutions with the same natural frequency
distribution by using a strategy similar to the one described in
\cite{C-T, LT}. We define a cumulative distribution function of a
density function $f$ for the KKE, say $F$, and we derive a new
integro-differential equation using its pseudo-inverse function.
Then, we use simple techniques for the optimal mass transport in
one-dimension, i.e., the equivalence relation between the
Wasserstein $p$-distance and the $L^p$-distance of the
corresponding pseudo-inverse of $F$ in order to obtain the
exponential decay estimate of the Wasserstein $p$-distance between
two measure-valued solutions. \newline

The rest of this paper is organized as follows. In Section 2, we
briefly review the Kuramoto model and its mean-field version (the
KKE), and we provide several a priori estimates. In Section 3, we
revisit an existence theory of measure valued solutions to the
KKE, and we present several a priori estimates, in particular, we
provide a finite-time stability estimate for measure-valued
solutions in a bounded Lipschitz distance. In Section 4, we show
the emergence of  completely synchronized states by lifting the
corresponding results for the KM to the KKE using the argument of
the particle-in-cell method. This strategy has been employed in
the Cucker-Smale flocking model in \cite{CFRT}. Section 5 is
devoted to the contraction property of the KKE using the method of
optimal mass transport as described in \cite{C-T, LT}. Finally,
the conclusions are stated in Section 6.

\newpage

%
%
%
%
\section{Preliminaries}
\setcounter{equation}{0}

In this section, we briefly review the particle Kuramoto model and
its kinetic mean-field model. Consider an ensemble of sinusoidally
coupled nonlinear oscillators that can be visualized as active
rotors on the circle $\bbs^1$. Throughout the paper, we will
identify a rotor with an oscillator. Let $x_i = e^{\sqrt{-1}
\theta_i}$ be the position of the $i$-th rotor. Then, the dynamics
of $x_i$ is completely determined by that of phase $\theta_i$. In
the absence of coupling, the phase equation for $\theta_i$ is
simply given by the decoupled ODE system:
\[
\frac{d\theta_i}{dt} = \Omega_i, \qquad \mbox{i.e.,} \quad \theta_i(t) = \theta_i(0) + \Omega_i t,
\]
where $\Omega_i$ is the natural phase-velocity (frequency) and is
assumed to be a random variable extracted from the density
function $g = g(\Omega)$:
\begin{align}
\begin{aligned} \label{dist}
& g(-\Omega) = g(\Omega), \quad \mbox{spt} (g)~\mbox{is bounded}~\\
& \int_{\bbr} \Omega g(\Omega) d\Omega = 0, \quad \int_{\bbr} g(\Omega) d\Omega = 1.
\end{aligned}
\end{align}
In the seminal work \cite{Ku2} of Kuramoto, he derived a coupled
phase model heuristically from the complex Ginzburg-Landau system. The KM is given by
\begin{equation} \label{Kuramoto}
\displaystyle \frac{d\theta_i}{dt} = \Omega_i - \frac{K}{N}
\sum_{j=1}^{N} \sin(\theta_i - \theta_j), \quad t >0, \quad i=1,
\cdots, N,
\end{equation}
subject  to initial data:
\begin{equation} \label{ini-p}
\theta_{i}(0) = \theta_{i0}.
\end{equation}
Note that the first term on the R.H.S. of \eqref{Kuramoto} represents
the intrinsic randomness, whereas the second term denotes the
nonlinear attractive coupling. Hence, synchronized states for
system \eqref{Kuramoto} will emerge, when the nonlinear coupling
dominates the intrinsic randomness. \newline

The system \eqref{Kuramoto} has been extensively studied over the last three
decades, and it remains a popular subject in nonlinear dynamics and
statistical physics (see review articles and a book \cite{A-B,
B-S, C-D, J-M-B, P-R-K, St}). In \cite{Ku1, Ku2}, Kuramoto first
observed that in the mean-field limit($N \to \infty$), the system
\eqref{Kuramoto} with a unimodal distribution function $g(\Omega)$
(which is assumed to be one-humped and symmetric with respect to
mean frequency $\Omega^p_{c} := \frac{1}{N} \sum_{i=1}^{N}
\Omega_i)$ has a continuous dynamical phase transition at a
critical value of the coupling strength $K_{cr}:$
\[
\displaystyle K_{cr} = \frac{2}{\pi g(0)},  \quad \mbox{in the
mean-field limit}.
\]
Moreover, he introduced an asymptotic order parameter $r^{\infty} \in [0, 1]$ to measure the degree of the phase
synchronization in mean-field limit:
\[
 r^{\infty}(K) := \lim_{t \to \infty} \lim_{N \to \infty} \Big| \frac{1}{N} \sum_{i=1}^{N} e^{\sqrt{-1} \theta_i(t)} \Big|,
\]
and he observed that this quantity $r^{\infty}$ changes from zero
to a non-zero value, when the coupling strength $K$ exceeds a
critical value $K_{cr}$. Note that for an initial phase
configuration that is uniformly distributed on the interval $[0,
2\pi)$, the quantity $r^{\infty}$ is exactly zero, whereas for a
completely synchronized configuration $\theta_i = \theta_c$ for
$i=1, \cdots, N$, $r^{\infty}$ becomes the unity. Therefore we can
regard $r^{\infty}$ as the ``{\it order parameter}'' measuring the
degree of synchronization. Before we conclude this subsection, we
recall a complete phase synchronization result from \cite{H-H-K}.
For this purpose, we introduce the diameters of the phase
configuration $\theta = (\theta_1, \cdots, \theta_N)$ and the
natural frequency set as follows.
\[
D_{\theta}(t) := \max_{1\leq i, j \leq N} |\theta_i(t) - \theta_j(t)|, \quad t \geq 0, \qquad D_{\Omega}:= \max_{i,j} |\Omega_i - \Omega_j |.
\]
\begin{proposition}\label{id-paricle}
\emph{\cite{H-H-K}} Suppose that the natural frequencies, the coupling strength and initial configuration satisfy
\[
\Omega_i = \Omega_j, \quad i \not = j, \quad K > 0, \quad D_0 := D_{\theta}(0) < \pi,
\]
and let $\theta = \theta(t)$ be the smooth solution to the system
\eqref{Kuramoto}-\eqref{ini-p} with initial phase configuration $\theta_0$. Then we have
\begin{equation} \label{decay}
\displaystyle e^{-Kt} D_0 \leq D_{\theta}(t) \leq
 e^{-K \alpha t} D_0, \quad t \geq 0,
\end{equation}
where $\alpha$ is the positive constant only depending on the
diameter of the initial phase configuration given by
\[\displaystyle  \alpha : = \frac{\sin D_0}{D_0}. \]
\end{proposition}

We also recall the estimate of existence of a trapping region for
non-identical oscillators from \cite{C-H-J-K} as follows.
\begin{lemma}\label{nid-particle}
\emph{\cite{C-H-J-K}} Let $\theta = \theta(t)$ be the global smooth
solution to \eqref{Kuramoto}-\eqref{ini-p} satisfying
\[
0 < D_0 < \pi, \quad D_{\Omega} > 0, \quad K> K_e := \frac{D_{\Omega}}{\sin D_0}.
\]
Then we have
\begin{eqnarray*}
&& (i)~\sup_{t \geq 0} D_{\theta}(t) \leq D_0 < \pi. \cr
&& (ii)~\exists~~t_0 > 0 \mbox{ such that} ~\sup_{t \geq t_0} D_{\theta}(t) \leq D^{\infty},
\end{eqnarray*}
where $D^{\infty}$ is defined by
\[
D^{\infty}:=\arcsin \Big[ \frac{D_\Omega}{K}\Big ] \in \Big( 0,
\frac{\pi}{2} \Big).
\]
\end{lemma}

\begin{remark}\label{cru-rmk-0}
If we set the average phase and natural frequency of the particles as
\[
\theta^p_c(t) := \frac{1}{N}\sum_{i=1}^{N}\theta_i(t),\quad \Omega^p_c := \frac{1}{N}\sum_{i=1}^{N}\Omega_i,
\]
then from the particle KM \eqref{Kuramoto}, one can easily obtain
\[
\theta_c^p(t) = \theta_c^p(0) + \Omega^p_c t, \quad \mbox{for all} \quad t\geq 0.
\]
Without loss of generality, we may assume that $\Omega^p_c = 0$ using
the phase-shift framework. Then we notice that identical and
non-identical oscillators satisfying the assumptions in
Proposition \ref{id-paricle} and Lemma \ref{nid-particle} satisfy
\begin{equation*}
\left\{
  \begin{array}{ll}
   \displaystyle \theta_i(t) \to \theta_c(0) \quad as \quad t \to \infty, \mbox{ for identical oscillators,}\\
   \displaystyle \theta_i(t) \in \Big(\theta^p_c(0) - D^{\infty},\theta^p_c(0) + D^{\infty} \Big) \quad \mbox{for all} \quad t\geq t_0, \mbox{ for non-identical oscillators.}
\end{array}
\right.
\end{equation*}
Note that the conditions and decay estimates \eqref{decay} are
independent of the particle-number $N$. For the related
synchronization estimates for the KM, we refer to \cite{C-H-J-K,
C-H-K-K, C-S, H-K, H-S, M-S1, M-S2}.
\end{remark}

We rewrite the system \eqref{Kuramoto} as a dynamical system on the extended phase space
$\bbt \times \bbr$ for $(\theta_i, \Omega_i):$ For $i=1, \cdots, N$,
\begin{equation} \label{re-ku}
\frac{d\theta_i}{dt} = \Omega_i - \frac{K}{N}
\sum_{j=1}^{N} \sin(\theta_i - \theta_j), \quad \frac{d\Omega_i}{dt} = 0, \qquad t > 0.
\end{equation}

Throughout this paper, we will use the interval $[0,2\pi)$ to denote $\bbt = \bbr / 2\pi \bbz$, i.e., $\theta \in [0,2\pi)$ implies that $\theta$
satisfies $\theta + 2\pi\bbz = \theta$.

%
%
%

\section{Existence theory of measure valued solutions}
\setcounter{equation}{0} In this section, we briefly review the
existence of measure valued solutions to \eqref{k-ku}. For the KM,
the rigorous mean-field limit was first done by Lancellotti
\cite{La} using Neunzert's general theory for the Vlasov equation
\cite{Ne, Sp}. Optimal transport arguments allow to generalize
these results in several ways for granular and flocking models
\cite{B-C-P, C-C-R}. H. Chiba recently obtained the same
mean-field limit based on functional tools \cite{Ch}. For a later
use and reader's convenience, we present several estimates for
measure valued solutions to the KKE.

\subsection{Measure-theoretic framework} In this subsection, we discuss a
measure-theoretic formulation of the KKE. \newline

Let ${\mathcal M}([0,2\pi) \times\bbr)$ be the set of nonnegative
Radon measures on $[0, 2\pi) \times \bbr$, which can be regarded
as nonnegative bounded linear functionals on ${\mathcal
C}([0,2\pi) \times \bbr)$. For a Radon measure $\nu \in {\mathcal
M}([0,2\pi) \times \bbr)$, we use the standard duality relation:
\[
\displaystyle   \langle \nu, h \rangle := \int_0^{2\pi} \int_{\bbr} h(\theta,\Omega) \nu(d\theta, d\Omega), \quad h \in \mathcal{C}_0([0,2\pi) \times\bbr).
\]

The definition of a measure-valued solution to equation
\eqref{k-ku} is given as follows.
\begin{definition}
For $T \in [0, \infty)$, let $\mu \in L^{\infty}([0, T); {\mathcal
M}([0,2\pi) \times\bbr))$ be a measure valued solution to \eqref{k-ku} with
an initial Radon measure $\mu_0 \in {\mathcal M}([0,2\pi) \times\bbr)$ if and only
if $\mu$ satisfies the following conditions:
\begin{enumerate}
\item
$\mu$ is weakly continuous:
\[ \langle \mu_t, h \rangle ~ \mbox{ is continuous as a function of $t$},~~\forall~ h \in \mathcal{C}_0([0,2\pi) \times\bbr). \]
\item
$\mu$ satisfies the integral equation: $\forall~h \in \mathcal{C}_{0}^{1}([0,2\pi)
\times\bbr\times [0,T))$,
\begin{equation} \label{def1}
 \langle \mu_t, h(\cdot,\cdot, t) \rangle - \langle \mu_0, h(\cdot,\cdot, 0) \rangle = \int_0^{t} \langle \mu_s, \partial_{s} h + \omega[\mu] \partial_{\theta} h \rangle ds,
\end{equation}
where $\omega[\mu](\theta,\Omega,s)$ is defined by
\begin{equation} \label{def2}
\omega[\mu](\theta,\Omega,s) := \Omega - K (\mu_s * \sin)\theta
\,.
\end{equation}
Here $*$ denotes the standard convolution, i.e.,
\[
(\mu_s * \sin)\theta = \int_0^{2\pi} \int_{\bbr} \sin(\theta - \theta_*)
\mu_s(d\theta_*,d\Omega).
\]
\end{enumerate}
\end{definition}
\begin{remark}
(i) Let $f = f(\theta, \Omega, t)$ be a classical solution to \eqref{k-ku}. Then the measure $\mu := f d\Omega
d\theta$ is a measure valued solution to \eqref{k-ku}. \newline

\noindent (ii) Note that the empirical measure
\[
\mu^N_t = \frac{1}{N} \sum_{i=1}^{N} \delta_{\theta_i(t)} \otimes
\delta_{\Omega_i(t)}, \quad \mbox{where $(\theta_i(t),
\Omega_i(t))$ is a  solution of \eqref{re-ku}}\,,
\]
is a measure valued solution to \eqref{k-ku} in the sense of
Definition 3.1. Thus, the solutions to the KM \eqref{Kuramoto} can
be treated as measure valued solutions via an empirical measure.
Here, $\delta_{z_*}$ is the Dirac measure concentrated at $z=z_*$.
\newline

\noindent (iii) Since the density function $g(\Omega)$ has a
compact support and the dynamics \eqref{k-ku} only governs the $\theta$-variable, we can see that the projected $\Omega$-support of
$\mu_t$ also has a compact support as well. Under the assumption of
the compact support of $g$, we can expand the class of test
functions to ${\mathcal C}([0,2\pi) \times \bbr)$. \newline

\noindent (iv) By choosing $h = \Omega$ in \eqref{def1} (see above
comment in (iii)), we have
\[ \langle \mu_t, \Omega \rangle =  \langle \mu_0, \Omega
\rangle, \quad t > 0. \]
\end{remark}
\begin{lemma}\label{conservation-lem}
Suppose that the density function $g =g(\Omega)$ has a compact
support and the initial measure satisfies
\[
\langle \mu_0, \Omega \rangle = 0,
\]
and let $\mu \in L^{\infty}([0, T); {\mathcal M}([0,2\pi) \times
\bbr))$ be a measure valued solution to \eqref{k-ku}. Then for $t \geq 0$, we have
\[
\langle \mu_t, 1 \rangle = \langle \mu_0, 1 \rangle=1, \qquad
\langle \mu_t, \theta \rangle = \langle \mu_0, \theta \rangle,
\qquad t \geq 0.
\]
\end{lemma}

\begin{proof} In \eqref{def1}, we set $h=1$. Then, the R.H.S. of \eqref{def1} will be zero hence, we have conservation of total mass.
For the time evolution of the first moment of $\theta$, it follows
from Remark 3.1 (iv) that
\[
\langle \mu_t, \Omega\rangle = 0, \quad t > 0.
\]
We now set $h(\theta)
= \theta$ in \eqref{def1} and use \eqref{def2} to get
\begin{eqnarray*}
\langle \mu_t, \theta \rangle  &=& \langle \mu_0, \theta \rangle + \int_0^t \langle \mu_s, \omega[\mu_s] \rangle ds \cr
                                 &=& \langle \mu_0, \theta \rangle + \int_0^t \Big( \langle \mu_s, \Omega \rangle - K \langle \mu_s, (\mu_s * \sin)\theta \rangle \Big) ds \cr
                                 &=& \langle \mu_0, \theta \rangle -K \int_0^t \langle \mu_s, (\mu_s * \sin)\theta \rangle ds
                                 =\langle \mu_0, \theta \rangle,
\end{eqnarray*}
where we used the anti-symmetry of $\sin(\theta - \theta_*)$ to
determine that $\langle \mu_s, (\mu_s * \sin)\theta \rangle = 0$.
\end{proof}

\subsection{A priori local stability estimate}
In this part, we recall the stability estimate for measure valued
solutions to \eqref{k-ku} in the bounded Lipschitz distance. This
stability estimate is crucial to the global existence of a measure
valued function for the KKE. First, we review
the definition of the bounded Lipschitz distance presented in \cite{H-L, Ne,
Sp}. We define the admissible set ${\mathcal S}$ of test
functions as
\[
{\mathcal S} := \Big \{ h: [0,2\pi) \times \bbr \to
\bbr~:~||h||_{L^{\infty}} \leq 1, \quad \mbox{Lip}(h) \leq 1 \Big
\}\,,
\]
where
$$
\mbox{Lip}(h) := \sup_{(\theta_1, \Omega_1) \not =  (\theta_2,
\Omega_2)} \frac{|h(\theta_1, \Omega_1) - h(\theta_2, \Omega_2)
|}{|(\theta_1, \Omega_1) - (\theta_2, \Omega_2)|}\,.
$$

\begin{definition}
\emph{\cite{Ne, Sp}} Let $\mu, \nu \in {\mathcal M}([0,2\pi) \times \bbr)$ be two Radon
measures. Then the bounded Lipschitz distance $d(\mu, \nu)$ between
$\mu$ and $\nu$ is given by
\[
\displaystyle d(\mu, \nu) := \sup_{h \in {\mathcal S}} \Big | \langle \mu,  h \rangle - \langle \nu,  h \rangle \Big|.
\]
\end{definition}

\begin{remark}
1. The space of Radon measures ${\mathcal M}([0,2\pi) \times \bbr)$ equipped with the
metric $d(\cdot, \cdot)$ is a complete metric space.
\newline

\noindent 2. The bounded Lipschitz distance $d$ for compactly
supported probability measures is equivalent to the Wasserstein-1
distance (Kantorovich-Rubinstein distance) $W_1$ (see \cite{Sp}):
\[
\displaystyle W_1(\mu, \nu) := \inf_{\gamma \in \Pi(\mu, \nu)}
\int_{0}^{2\pi}\int_\bbr \int_{0}^{2\pi}\int_\bbr |(\theta
-\theta_*,\Omega-\Omega_*)| \,\gamma(d(\theta,\Omega),
d(\theta_*,\Omega_*)),
\]
where $\Pi(\mu, \nu)$ is the set of all product measures on
$([0,2\pi) \times\bbr) \times ([0,2\pi) \times\bbr)$ such that
their marginals are $\mu$ and $\nu,$ respectively. Both equivalent
distances endow the weak-$*$ convergence of measures with metric
structure in bounded sets.
\newline

\noindent 3. For any $h \in \mathcal{C}([0,2\pi) \times \bbr)$ with
$\|h\|_{L^{\infty}} \le a$ and $\mbox{Lip}(h) \le b$, we have
\begin{equation*}
\Big | \langle \mu, h \rangle - \langle \nu, h \rangle \Big| \leq \max \{a, b \} d(\mu, \nu).
\end{equation*}
\end{remark}

%
%
%
%

\subsection{Existence of a measure valued function for KKE}
We briefly present the existence of a measure valued solution to
the KKE using a Lancellotti's argument \cite{La}. We first note a
simple approximation argument of measures by smooth positive
densities. For given $\mu_0 \in
L^{\infty}(0,T;\mathcal{M}([0,2\pi) \times \bbr))$, we set for
$0 < \varepsilon < 1$,
\[
\tilde{\mu}_0^{\varepsilon} := ((1-\varepsilon)\mu_0+\varepsilon
\chi_R)
* \eta_{\varepsilon} = \int_{[0,2\pi] \times \bbr}
\eta^1_{\varepsilon}(\theta-\tilde\theta)
\,\eta^2_{\varepsilon}(\Omega-\tilde\Omega)
\,((1-\varepsilon)\mu_0+\varepsilon
\chi_{\mathcal{R}})(d\tilde\theta,d\tilde\Omega),
\]
where $\eta_{\varepsilon} = (\eta^1_{\varepsilon}, \eta^2_{\varepsilon})$, $\eta^1_{\varepsilon}$ is a periodic compactly supported mollifier with period $2\pi$ and $\eta^2_{\varepsilon}$ is
a standard compactly supported mollifier satisfying
$\|\eta^1_{\varepsilon}\|_{L^1([0,2\pi])} = 1 $, $\|\eta^2_{\varepsilon}\|_{L^1(\bbr)} = 1$, and
$\chi_{\mathcal{R}}$ is the uniform probability measure on rectangle $\mathcal{R}$ enclosing $\mbox{spt}(\mu_0)$ such that
$\mathcal{R} \subset [0,2\pi] \times [-C,C]$ with $C>0$. It is
straightforward to check that
\[
d(\tilde{\mu}_0^{\varepsilon},\mu_0)\simeq
W_1(\tilde{\mu}_0^{\varepsilon},\mu_0) \to 0, \quad \mbox{as}
\quad \varepsilon \to 0
\]
and
$\mbox{spt}(\tilde\mu_0^\varepsilon) \subset [0,2\pi] \times [-C_\varepsilon,C_\varepsilon]$.

We remark that if $\mbox{spt}(\mu_0) \subset (0,2\pi) \times [-C,C]$, then $\mbox{spt}(\tilde\mu_0^\varepsilon) \subset (0,2\pi) \times [-C_{\varepsilon},C_{\varepsilon}]$ using a standard compactly supported mollifier $\eta_{\varepsilon}$ for sufficiently small $\varepsilon$.
Since $\tilde{\mu}_0^{\varepsilon}$ is absolutely continuous with
respect to Lebesgue measure $d\theta d\Omega$ and with connected
support, then for all $\varepsilon$ the initial approximated
measure $\tilde{\mu}^{\varepsilon}_0$ can be approximated by a
Dirac comb of uniform masses. This means that there exist point
distributions $\{(\theta^{\varepsilon,N}_{i0},\Omega^{\varepsilon,N}_{i0})\}_{i=1,...,N}$, whose
dependence on $N$ is elapsed for clarity, such that
\begin{equation}\label{inital_aprox}
\lim_{N\to\infty} d(\tilde{\mu}^{\varepsilon}_0,
\tilde{\mu}^{\varepsilon,N}_0) = 0,\qquad
\tilde{\mu}^{\varepsilon,N}_0 := \frac{1}{N} \sum_{i=1}^{N}
\delta_{\theta^{\varepsilon,N}_{i0}} \otimes \delta_{\Omega^{\varepsilon,N}_{i0}}.
\end{equation}
Then we solve the KM with $N$-particles:
\begin{equation}\label{pic}
   \frac{d\theta^{\varepsilon,N}_i}{dt} = \Omega^{\varepsilon,N}_i -
\frac{K}{N}\sum_{j=1}^{N}\sin(\theta^{\varepsilon,N}_i - \theta^{\varepsilon,N}_j), \quad \frac{d\Omega^{\varepsilon,N}_i}{dt} = 0, \quad t > 0, \quad i=1, \cdots, N,
\end{equation}
with initial data:
\[
(\theta^{\varepsilon,N}_i(0), \Omega^{\varepsilon,N}_{i}(0)) = (\theta^{\varepsilon,N}_{i0}, \Omega^{\varepsilon,N}_{i0}).
\]
With solutions $(\theta^{\varepsilon,N}_i(t),\Omega^{\varepsilon,N}_i(t))$ of \eqref{pic}, the approximate solution $\tilde{\mu}^{\varepsilon,N}_t$ for the measure
valued solution can be constructed as a sum of Dirac measures, i.e.,
\begin{equation}\label{pic-t}
\tilde{\mu}^{\varepsilon,N}_t := \frac{1}{N}\sum_{i=1}^{N}\delta_{\theta^{\varepsilon,N}_i(t)} \otimes \delta_{\Omega^{\varepsilon,N}_i(t)}.
\end{equation}

The results of \cite{La,Ne,Sp} imply the
continuous dependence with respect to initial data for measure
valued solutions. Therefore, there exists a constant $C$ depending
on $T$ and the initial support of the measure $\mu_0^1,\mu_0^2$
such that
\begin{equation}\label{cont-dep}
d(\mu^1_t,\mu^2_t) \leq C \,d(\mu^1_0,\mu^2_0).
\end{equation}
for all $t\in[0,T]$. As a consequence of \eqref{cont-dep}, we have
the convergence of approximate solutions and particle
approximations to the measure valued solution of \eqref{k-ku}:
\begin{equation*}
d(\tilde{\mu}^{\varepsilon}_t,\tilde{\mu}^{\varepsilon,N}_t) \leq
C\, d(\tilde{\mu}^{\varepsilon}_0,\tilde{\mu}^{\varepsilon,N}_0)
\qquad \mbox{and} \qquad d(\tilde{\mu}^{\varepsilon}_t,\mu_t) \leq
C\, d(\tilde{\mu}^{\varepsilon}_0,\mu_0).
\end{equation*}
This obviously yields
$$
d(\tilde{\mu}^{\varepsilon,N}_t,\mu_t) \leq C \big(
d(\tilde{\mu}^{\varepsilon,N}_0, \tilde{\mu}^{\varepsilon}_0) +
d(\mu_0, \tilde{\mu}^{\varepsilon}_0)\big)\,.
$$
Hence we first let $\varepsilon \to
0$ and then letting $N \to \infty$ to have
\[
d(\tilde{\mu}^{\varepsilon,N}_t,\mu_t) \to 0.
\]
Therefore we have the following theorem.

\begin{theorem}\label{local-stability}
Suppose that the initial measure $\mu_0 \in {\mathcal M}([0,2\pi)
\times \bbr)$ and let $\mu^{N}_t$ be the approximate solution
constructed by \eqref{pic-t}. Then there exists a unique measure
valued solution $\mu_t \in {\mathcal M}([0,2\pi) \times \bbr)$ to
\eqref{k-ku}  such that $\mu_t$ is the weak-$*$ limit of the
approximate solutions $\mu_t^{N}$ as $N \to \infty$, i.e.,
\[
\lim_{N\to\infty}d(\mu_t, \mu_t^{N})  = 0.
\]
\end{theorem}
From now on, let us assume that the initial measure is a smooth
absolutely continuous measure with respect to Lebesgue with
connected support. These assumptions can be eliminated by standard
mollifier approximation as above. Therefore, we will proceed by
working on smooth solutions and obtaining estimates depending only
on quantities that pass to the limit in the weak-$*$ sense, and
thus, stable estimates under this approximation. To avoid too much
repetition, this procedure will not be specified in the proofs
below and the statements of the results will be written directly
for measure valued solutions.

%
%
%
\section{Asymptotic complete synchronization estimate}
\setcounter{equation}{0} In this section, we present an asymptotic
synchronization estimate for the KKE \eqref{k-ku} by lifting
corresponding results for the KM \eqref{Kuramoto} using the
argument of the particle-in-cell method \cite{Ra} discussed in the
previous section.

Let $\mu \in L^{\infty}([0, T); {\mathcal M}([0,2\pi) \times \bbr))$ be a measure valued solution to \eqref{k-ku}, and let $R(t)$ and $P(t)$
be the orthogonal $\theta$ and $\Omega$-projections of $\mbox{spt}
(\mu_t)$ respectively, i.e.,
\[ R(t) := \bbp_{\theta} \mbox{spt}(\mu_t), \qquad  P(t) := \bbp_{\Omega} \mbox{spt}(\mu_t), \]
Then it is easy to see that
\[
P(t) = P(0), \qquad t \geq 0.
\]
We also set
\begin{align*}
\begin{aligned}
D_{\theta}(\mu_t) &:= \mbox{diam}(R(t)), \quad
D_{\Omega}(\mu_t) := \mbox{diam}(P(t)), \quad M(t):= \langle \mu_t, 1 \rangle, \\
\theta_c(t) &:= \frac{1}{M(t)}\langle \mu_t, \theta \rangle, \quad \Omega_c(t) := \frac{1}{M(t)} \langle \mu_t, \Omega \rangle.
\end{aligned}
\end{align*}
We observe from Lemma \ref{conservation-lem} that
\begin{equation*}\label{omega-conserve}
M(t) = \langle \mu_t, 1 \rangle = \langle \mu_0, 1 \rangle =
M(0)=1,
\end{equation*}
and since $\Omega_c(0)=0,$ we obtain
\[
\Omega_c(t) = 0 \quad \mbox{and} \quad \theta_c(t) = \theta_c(0), \quad t \geq 0.
\]
\begin{lemma}\label{id-mea-val-sol}
Suppose that the oscillators are identical and let $\mu_0$ be a
given initial probability measure in ${\mathcal M}([0, 2\pi)
\times \bbr)$ satisfying
\[
\langle \mu_0, \theta \rangle  = \pi, \quad D_{\theta}(\mu_0) < \pi, \quad K > 0.
\]
Then the measure valued solution $\mu$  to \eqref{k-ku} - \eqref{ini} with an initial
datum $\mu_0$ satisfies
\[
D_{\theta}(\mu_t) \leq D_{\theta}(\mu_0)e^{-K{\bar \alpha} t}, \qquad t \geq 0,
\]
where
\[
{\bar \alpha} = \frac{\sin D_{\theta}(\mu_0)}{D_{\theta}(\mu_0)}.
\]
\end{lemma}
\begin{proof}
We use the approximation argument in Theorem \ref{local-stability}
giving a initial particle approximation $\mu^N_0$ defined as in
\eqref{inital_aprox}. Then, it follows from Proposition
\ref{id-paricle} that the approximate measure valued solution
$\mu^N_t \in {\mathcal M}([0, 2\pi) \times \bbr)$ satisfies
\[
D_{\theta}(\mu^N_t) \leq D_{\theta}(\mu^N_0) e^{-K {\bar \alpha_N}
t}, \quad t \geq 0,
\]
where
\[
{\bar \alpha}_N = \frac{\sin
D_{\theta}(\mu^N_0)}{D_{\theta}(\mu^N_0)}.
\]
Since Theorem \ref{local-stability} implies that $
d(\mu_t,\mu^{N}_t)  \rightarrow 0$ as $\quad N \to \infty$.
Hence, $D_{\theta}(\mu^N_t)\to D_{\theta}(\mu_t)$ as $N\to\infty$
and we obtain the desired result.
\end{proof}

\begin{remark}\label{cru-rmk}
Throughout the paper, without loss of generality, we assume that
$\langle \mu_0,\theta \rangle = \pi$ in order to avoid any possible
confusion arising from the periodicity of $\theta$. In fact, if
the oscillators satisfy the assumption in Lemma
\ref{id-mea-val-sol} (or Lemma \ref{lemma-appli-1}), the orthogonal $\theta$-projection of
spt($\mu_t$), $R(t)$ is confined to the interval $(0,2\pi)$ for all
$t \geq 0$ (see Remark \ref{cru-rmk-0}). This property will also be
significantly used in Section 5 (see Lemma \ref{lemma-appli-2}).
\end{remark}

We now show that the measure valued solution to the system \eqref{k-ku}
for identical oscillators will converge to a multiple of the Dirac
measure concentrated on $\left( \theta_c(0), \Omega_c(0) \right)$ in the phase space  $(\theta,
\Omega)$. We set
\[
\mu_{\infty}(d\theta, d\Omega) :=  \delta_{\theta_c(0)}(\theta) \otimes \delta_{\Omega_c(0)}(\Omega).
\]

\begin{theorem}
Suppose that the oscillators are identical, and let $\mu_0 \in
{\mathcal M}([0, 2\pi) \times \bbr)$ be a given initial
probability measure satisfying
\[
\theta_c(0) = \pi, \quad  D_{\theta}(\mu_0) < \pi \quad \mbox{and} \quad K > 0.
\]
Then the measure valued solution $\mu_t$ to \eqref{k-ku} with initial
datum $\mu_0$ satisfies
\[ \displaystyle  \lim_{t \to \infty} d(\mu_t, \mu_{\infty}) = 0, \]
where $d=d(\cdot, \cdot)$ is the bounded Lipschitz distance defined
in Section 3.
\end{theorem}
\begin{proof}
Let $h = h(\theta) \in {\mathcal C}([0,2\pi))$ be any test function
satisfying
\[ \|h\|_{L^{\infty}} \leq 1, \qquad \mbox{Lip}(h) \leq 1. \]
Then $h$ can also be regarded as a test function in ${\mathcal
C}([0,2\pi) \times \bbr)$.
Then we have
\begin{eqnarray*}
&& \Big| \int_{[0,2\pi] \times \bbr} h(\theta)
\mu_t(d\theta, d\Omega) - \int_{[0,2\pi] \times \bbr} h(\theta)
\mu_{\infty}(d\theta, d\Omega) \Big|  \cr
 && \hspace{1cm} = \Big|\int_{0}^{2\pi} h(\theta)
{\bar \mu}_t(d\theta) - h(\pi) \Big| \leq \int_{0}^{2\pi} |\theta-\pi| {\bar \mu}_t(d\theta) \leq D_{\theta}(\mu_0)e^{-K {\bar \alpha} t}.
\end{eqnarray*}
where we used Lemma \ref{id-mea-val-sol}, and ${\bar \mu}_t(d\theta)$ is a
$\theta$-marginal of the measure $\mu_t$, i.e.,
\[ {\bar \mu}_t (d\theta) := \int_{\bbr} \mu_t(d\theta, d\Omega).
\]
This implies that
\[ d(\mu_t, \mu_{\infty}) \leq D_{\theta}(\mu_0)e^{-K{\bar \alpha} t} \to 0, \qquad \mbox{as}~~t \to \infty. \]
\end{proof}
%
%
\section{Stability estimate of the KKE}
\setcounter{equation}{0}
In this section, we present the strict contractivity of measure
valued solutions to the KKE by using the method of optimal mass
transport \cite{C-T,LT,Vi}. The strict contractivity result
generalizes the $L^1$-contraction result for the KM in
\cite{C-H-J-K}.

\subsection{Alternative formulation of the KKE}
In this part, we derive an alternative form of the KKE, which is
more convenient for deriving estimates in terms of the Wasserstein-distance.
First, we study the existence of an invariant set for the KKE.

\begin{lemma}\label{lemma-appli-1}
Suppose that the initial probability measure $\mu_0$ and the
coupling strength $K$ satisfy
\[
0<D_{\theta}(\mu_0) <\pi, \quad  0< D_{\Omega}(\mu_0) <\infty, \quad K > \frac{D_{\Omega}(\mu_0)}{\sin D_{\theta}(\mu_0)}.
\]
Then, there exist $t_0>0$ and $D^{\infty} \in (0, \frac{\pi}{2})$
such that the measure valued solution $\mu$ to \eqref{k-ku} with
initial datum $\mu_0$ satisfies
\[
D_{\theta}(\mu_t) \leq D^{\infty}, \quad t \ge t_0.
\]
\end{lemma}
\begin{proof}
We apply the argument similar to that in the proof of Lemma
\ref{id-mea-val-sol}. Let $N>0$ be given. Then we have the
following approximation $\mu_0^N$ for $\mu_0$:
\[
\mu_0^N=\frac{1}{N}\sum_{i=1}^{N}\delta_{\theta_{i0}} \otimes
\delta_{\Omega_{i0}}.
\]
We now solve the Cauchy problem for KM:
\begin{equation*}
\left\{
  \begin{array}{ll}
   \displaystyle \frac{d\theta_i}{dt} = \Omega_i + \frac{K}{N}
\sum_{j=1}^{N}\sin(\theta_j - \theta_i), \quad t > 0,\\
   \displaystyle \frac{d\Omega_i}{dt} = 0.
\end{array}
\right.
\end{equation*}
subject to initial data $(\theta_i(0),\Omega_i(0)) = (\theta_{i0},
\Omega_{i0})$. Theorem \ref{local-stability} implies that
\[
d(\mu_t,\mu^{N}_t)  \rightarrow 0 \quad \mbox{as} \quad N \to
\infty\,,
\]
and thus, $D_{\theta}(\mu^N_t)\to D_{\theta}(\mu_t)$ and
$D_{\Omega}(\mu^N_t)\to D_{\Omega}(\mu_t)$ as $N\to\infty$.
Hence we can take $N$ large enough such that
$D_{\Omega}(\mu^N_0)$ and $D_{\theta}(\mu^N_0)$ satisfies the
conditions of Lemma \ref{nid-particle}. Thus, we find that there
exist $t_0^N > 0$ and $D^{\infty,N}$ such that
\[
D_{\theta}(\mu^N_t) \le D^{\infty,N},\quad t\ge t_0^N, \quad \mbox{for $N$ large enough},
\]
where
\[
t_0^N := \frac{D_{\theta}(\mu_0^N) - D^{\infty,N}}{ K \sin D_{\theta}(\mu_0^N) -  D_{\Omega}(\mu_0^N)}, \quad D^{\infty,N} := \arcsin \big[ \frac{D_{\Omega}(\mu_0^N)}{K}\big] \in \big( 0, \frac{\pi}{2}\big).
\]
We now let $N \to \infty$ to obtain the
desired result.
\end{proof}

In the remainder of this section, from Remark \ref{cru-rmk}, we
assume that
\begin{equation}\label{m-assum-1}
R(t) \subset \big(0, 2\pi\big) \quad \mbox{and} \quad t \geq 0.
\end{equation}

Under this assumption the solution is given by a smooth particle
density function $f(\theta, \Omega, t)$ in $L^1$ for all $t\geq
0$. For a given $\Omega$, we consider a one-particle density
function $f$ as a function of $\theta$. Then we define the pseudo
cumulative distribution function of $f$:
\[
F(\theta, \Omega, t) := \int_{0}^{\theta} f(\theta_*,\Omega,t) d\tilde{\theta}, \qquad (\theta, \Omega, t) \in [0,2\pi) \times \bbr \times \bbr_+,
\]
and a pseudo-inverse $\phi$ of $F(\cdot, \Omega, t)$ as a function
of $\theta$:
\begin{equation*}
\phi(\eta, \Omega,t) := \inf \{\theta :
F(\theta, \Omega, t) > \eta\},\quad \eta \in [0,g(\Omega)].
\end{equation*}
As long as there is no confusion, we use the notation $F^{-1}(\eta,
\Omega, t) = \phi$ as the pseudo inverse of $F$ as $\theta$-function. Then it
is easy to see that
\begin{equation} \label{inverse-rel}
F(\phi(\eta, \Omega, t), \Omega, t) = \eta.
\end{equation}

\begin{lemma}\label{lemma-appli-2}
Let $\mu$ be a measure-valued solution to
\eqref{k-ku}-\eqref{ini}, and let $\phi$ be the pseudo-inverse
function of the cumulative distribution function $F$. Then we have
\begin{eqnarray*}
&& (i)~\max\{\theta~|~\theta\in R(t) \}=\max_{\Omega\in
\mbox{spt}(g)}\phi(g(\Omega),\Omega,t). \cr
&& (ii)~\min\{\theta~|~\theta\in R(t) \}=\min_{\Omega\in \mbox{spt}(g)}\phi(0,\Omega,t). \cr
&& (iii)~\max_{\Omega\in \mbox{spt}(g)}\phi(g(\Omega),\Omega,t)-\min_{\Omega\in \mbox{spt}(g)}\phi(0,\Omega,t)\le D^{\infty},\quad t\ge t_0.
\end{eqnarray*}
\end{lemma}

\begin{proof}
Since the estimate for $(ii)$ is similar to that of $(i)$ and the
estimate for (iii) follows from the estimates (i) and (ii), we
only provide the proof for the estimate (i). For notational
simplicity, we set
\[ \theta_M := \max\{\theta~|~\theta\in R(t) \}.  \]
Then, by definition of $\mu_t$, we have
\[
\theta_M = \max\{\theta~|~\theta\in
\mbox{spt}_{\theta}(f(\theta,\Omega,t)) \mbox{ and }~\Omega\in
\mbox{spt}(g)\},
\]
where $\mbox{spt}_{\theta}(f(\theta,\Omega,t))$ is the
$\theta$-projection of  $\mbox{spt}(f(\theta,\Omega,t))$.
This yields
\[
\theta_M = \max\{\phi(g(\Omega),\Omega,t) \mbox{ such that
$\Omega\in \mbox{spt}(g)$} \},
\]
by definition of  the pseudo-inverse function. This completes the proof.
\end{proof}

Next, we derive an integro-differential equation for the pseudo
inverse $\phi$. It follows from \eqref{m-assum-1} that the smooth
solution $f(\theta, \Omega, t)$ to \eqref{k-ku}-\eqref{ini}
satisfies
\begin{equation*} \label{BC}
f(0, \Omega, t) = 0, \quad \Omega \in \bbr,~t \geq 0.
\end{equation*}
We differentiate the relation \eqref{inverse-rel} in $t$ and use
$\partial_{\theta} F = f$ to get
\[ \partial_t F(\theta, \Omega, t) \Big|_{\theta =
\phi(\eta, \Omega, t)} + f(\theta, \Omega, t)\Big|_{\theta =
\phi(\eta, \Omega, t)} \partial_t \phi(\eta, \Omega,
t) = 0. \]
This yields
\begin{align*}
\partial_t \phi(\eta, \Omega,t) & = -\frac{1}{f(\theta, \Omega,
t)} \partial_t F(\theta, \Omega,t) \Big|_{\theta=\phi(\eta,
\Omega,t)} \cr & = \frac{1}{f(\theta, \Omega, t)}\Big|_{\theta =
\phi(\eta, \Omega, t)} \times (\omega[f] f)(\cdot, \Omega, t)
\Big|_{\theta = 0}^{\theta = \phi(\eta, \Omega, t)} \cr & =
\Omega+ K\int_{\bbr} \int_0^{2\pi} \sin(\theta_*- \phi(\eta,
\Omega,t))f(\theta_*,\Omega_*, t) d\theta_*  d\Omega_* \quad
\mbox{using \eqref{BC}} \cr & = \Omega + K \int_{\bbr} \int_0^{g(\Omega_*)}
\sin (\phi(\eta_*, \Omega_*,t)-\phi(\eta, \Omega,t))
d\eta_* d\Omega_*,
\end{align*}
where we used $\theta_* = \phi(\eta_*, \Omega_*, t)$ and relation
\eqref{inverse-rel} to see $f(\theta_*, \Omega_*,t) d\theta_* =
d\eta_*$. Hence, the pseudo-inverse $\phi$ satisfies the following
integro-differential equation:
\begin{equation} \label{phi-t}
\partial_t \phi = \Omega + K\int_{\bbr}  \int_0^{g(\Omega_*)} \sin (\phi_*-\phi) d\eta_* d\Omega_* .
\end{equation}
where we used abbreviated notations:
\[ \phi_* := \phi(\eta_*,\Omega_*,t), \quad \phi := \phi(\eta,\Omega,t). \]
The following results is a simple consequence of the change of
variables and Lemma \ref{conservation-lem}.

\begin{lemma}\label{contract-lemma1}
Let $\mu_t$ be a measure valued solution to \eqref{k-ku} -
\eqref{ini} with an associated pseudo-inverse function $\phi$.
Then, we have
\[ \int_{\bbr} \int_0^{g(\Omega)}  \phi  d\eta d\Omega = \int_{\bbr} \int_{0}^{2\pi}  \theta
\mu_t(d\theta, d\Omega), \quad \frac{d}{dt} \int_{\bbr} \int_0^{g(\Omega)} \phi d\eta d\Omega = 0. \]
\end{lemma}

%
%
\subsection{Strict contractivity in the Wasserstein distance}
In this part, we present the proof of the strict contraction property of the KKE. \newline

For the one-dimensional case, it is well known \cite{C-T, Vi} that
the Wasserstein $p$-distance $W_p(\mu_1, \mu_2)$ between two
measures $\mu_1$ and $\mu_2$ is equivalent to the $L^p$-distance
between the corresponding pseudo-inverse functions $\phi_1$ and
$\phi_2$ respectively. Thus, we set
\[
W_p(\mu_1,\mu_2)(\Omega, t):=
\|\phi_1(\cdot,\Omega,t)-\phi_2(\cdot,\Omega,t)\|_{L^p(0,g(\Omega))},
\quad 1\leq p \le \infty.
\]
Since $W_p(\mu_1,\mu_2)$ depends on $\Omega$, we introduce a
modified metric on the phase-space $(\theta,\Omega)$:
\begin{equation*}
\widetilde{W}_p(\mu_1,\mu_2)(t) := ||W_p(\mu_1,\mu_2)(\cdot,
t)||_{L^p(\bbr)}, \quad 1\leq p \leq \infty.
\end{equation*}
Below, we assume that the density function $g(\Omega)$ has compact
support. Then, it is easy to see that
$\widetilde{W}_p(\mu_1,\mu_2)$ is a metric that satisfies
\begin{equation}\label{limit}
\lim_{p \rightarrow \infty} \widetilde{W}_p(\mu_1,\mu_2)(t) =
\widetilde{W}_{\infty}(\mu_1,\mu_2)(t), \qquad t \geq 0.
\end{equation}
Recall that the $\mbox{sgn}$ function is defined by
\[ \mbox{sgn}(x) = \left\{
                     \begin{array}{ll}
                       1, & x > 0,  \\
                       0, & x = 0,  \\
                       -1, & x < 0.
                     \end{array}
                   \right.
\]
\begin{lemma}\label{lemma-cal}
Let $\Phi$ be a measurable function defined on $[0,g(\Omega)]\times \bbr$ satisfying
\[
 |\Phi(\eta,\Omega)| < \frac{\pi}{2} \quad \mbox{and} \quad \int_{\bbr} \int_0^{g(\Omega)} \Phi(\eta,\Omega) d\eta d\Omega =0.
\]
Then for $1\leq p < \infty$, we have
\begin{align*}
\begin{aligned}
& \int_{\bbr} \int_{\bbr} \int_0^{g(\Omega)}\int_0^{g(\Omega_*)}  \Big[|\Phi(\eta,\Omega)|^{p-1} \mbox{sgn}(\Phi(\eta,\Omega))-|\Phi(\eta_*,\Omega_*)|^{p-1} \mbox{sgn}(\Phi(\eta_*,\Omega_*)) \Big ] \cr
& \hspace{1cm} \times \sin
(\frac{\Phi(\eta_*,\Omega_*)-\Phi(\eta,\Omega)}{2})  d\eta_* d\eta d\Omega_* d\Omega \leq -\frac{2}{\pi}\int_{\bbr} \int_0^{g(\Omega)} |\Phi(\eta)|^p d\eta d\Omega.
\end{aligned}
\end{align*}
\end{lemma}
\begin{proof} For notational simplicity, we set
\begin{eqnarray*}
&& \Phi := \Phi(\eta,\Omega),~~\Phi_* := \Phi(\eta_*,\Omega_*), \quad \mbox{and} \cr
&& \Delta(\eta, \eta_*,\Omega,\Omega_*) := \Big[|\Phi|^{p-1}\mbox{sgn}(\Phi)-
|\Phi_*|^{p-1}\mbox{sgn}(\Phi_*)\Big ] \sin \Big(\frac{\Phi_*-\Phi}{2} \Big),
\end{eqnarray*}
and we decompose the domain $[0,g(\Omega)] \times \bbr$ as the disjoint union of three subsets:
\[ {\mathcal P}:=\{(\eta,\Omega)~|~\Phi(\eta, \Omega) > 0 \}, \quad {\mathcal Z}:=\{(\eta,\Omega)~|~\Phi(\eta, \Omega) = 0 \}, \quad  {\mathcal N}:=\{(\eta,\Omega)~|~ \Phi(\eta, \Omega) < 0 \}.\]
Then it follows from the condition $\int_{\bbr} \int_0^{g(\Omega)} \Phi d\eta d\Omega= 0$ that
\begin{equation} \label{zero}
 \int_{{\mathcal P}} |\Phi| d\eta d\Omega=  \int_{{\mathcal N }} |\Phi| d\eta d\Omega.
 \end{equation}
We use $[0,g(\Omega)] \times \bbr =  {\mathcal P} \cup  {\mathcal Z} \cup  {\mathcal N}$ to obtain
\begin{align*}
\begin{aligned}
& \int_{\bbr} \int_{\bbr} \int_0^{g(\Omega_*)}\int_0^{g(\Omega)} \Delta(\eta, \eta_*,\Omega,\Omega_*) d\eta d\eta_* d\Omega d\Omega_*\\
& \hspace{0.5cm} = \Big( \underbrace{\int_{{\mathcal P} \times {\mathcal Z}} + \cdots  \int_{{\mathcal N} \times {\mathcal P}}}_{\mbox{distinct signs}} +
\underbrace{\int_{{\mathcal P} \times {\mathcal P}} + \int_{{\mathcal N} \times {\mathcal N}} + \int_{{\mathcal Z} \times {\mathcal Z}}}_{\mbox{same signs}}  \Big)
\Delta(\eta, \eta_*,\Omega,\Omega_*) d\eta d\eta_* d\Omega d\Omega_*
\end{aligned}
\end{align*}
We now consider the following sub-integrals separately.
\[ I(A,B) := \int_{A \times B} \Delta(\eta, \eta_*,\Omega,\Omega_*) d\eta d\eta_* d\Omega d\Omega_*, \quad A, B \in \{{\mathcal P}, {\mathcal Z}, {\mathcal N} \}. \]
We claim the following:
\begin{center}
\begin{tabular}{||c|c|c|c||}
\hline \hline Case   & A & B & $I(A, B) \leq$ \\
\hline \hline I & ${\mathcal P}$ & ${\mathcal Z}$ & $-\frac{{\mathcal L}({\mathcal Z})}
{\pi}  \int_{{\mathcal P}}|\Phi_*|^p d\eta_*d\Omega_*$ \\
\hline II  & ${\mathcal N}$ & ${\mathcal Z}$ & $-\frac{{\mathcal L}({\mathcal Z})}
{\pi}  \int_{{\mathcal N}}|\Phi_*|^p d\eta_*d\Omega_*$ \\
\hline III & ${\mathcal Z}$ & ${\mathcal P}$ & $-\frac{{\mathcal L}({\mathcal Z})}
{\pi}  \int_{{\mathcal P}}|\Phi|^p d\eta d\Omega$ \\
\hline IV & ${\mathcal Z}$ & ${\mathcal N}$ & $-\frac{{\mathcal L}({\mathcal Z})}
{\pi}  \int_{{\mathcal N}}|\Phi|^p d\eta d\Omega$  \\
\hline V & ${\mathcal P}$ & ${\mathcal N}$ & $ -\frac{1}{\pi}\Big [ {\mathcal L}({\mathcal P}) \int_{\mathcal N} |\Phi|^p d\eta d\Omega
+ {\mathcal L}({\mathcal N})
\int_{\mathcal P} |\Phi_*|^p d\eta_* d\Omega_*$ \\
   &   &  &$+ \int_{\mathcal N} |\Phi|^{p-1}d\eta d\Omega \int_{\mathcal P} |\Phi_*|d\eta_* d\Omega_* + \int_{\mathcal P} |\Phi_*|^{p-1}d\eta_* d\Omega_*
   \int_{\mathcal N} |\Phi|d\eta d\Omega \Big ] $ \\
\hline VI & ${\mathcal N}$ & ${\mathcal P}$ & $ -\frac{1}{\pi}\Big [ {\mathcal L}({\mathcal N}) \int_{\mathcal P} |\Phi|^p d\eta d\Omega
+ {\mathcal L}({\mathcal P})
\int_{\mathcal N} |\Phi_*|^p d\eta_* d\Omega_*$ \\
   &   &  &$+ \int_{\mathcal P} |\Phi|^{p-1}d\eta d\Omega \int_{\mathcal N} |\Phi_*|d\eta_* d\Omega_* + \int_{\mathcal N} |\Phi_*|^{p-1}d\eta_* d\Omega_*
   \int_{\mathcal P} |\Phi|d\eta d\Omega \Big ] $ \\
\hline VII& ${\mathcal P}$ & ${\mathcal P}$ & $ -\frac{1}{\pi}\Big [ {\mathcal L}({\mathcal P}) \int_{\mathcal P} |\Phi|^p d\eta d\Omega
+ {\mathcal L}({\mathcal P})
\int_{\mathcal P} |\Phi_*|^p d\eta_* d\Omega_*$ \\
   &   &  &$- \int_{\mathcal P} |\Phi|^{p-1}d\eta d\Omega \int_{\mathcal P} |\Phi_*|d\eta_* d\Omega_* - \int_{\mathcal P} |\Phi_*|^{p-1}
   d\eta_* d\Omega_* \int_{\mathcal P} |\Phi|d\eta d\Omega \Big ] $ \\
\hline VIII& ${\mathcal N}$ & ${\mathcal N}$ & $ -\frac{1}{\pi}\Big [ {\mathcal L}({\mathcal N}) \int_{\mathcal N} |\Phi|^p d\eta d\Omega
+ {\mathcal L}({\mathcal N})
\int_{\mathcal N} |\Phi_*|^p d\eta_* d\Omega_*$ \\
   &   &  &$- \int_{\mathcal N} |\Phi|^{p-1}d\eta d\Omega \int_{\mathcal N} |\Phi_*|d\eta_* d\Omega_* - \int_{\mathcal N} |\Phi_*|^{p-1}
   d\eta_* d\Omega_* \int_{\mathcal N} |\Phi|d\eta d\Omega \Big ] $ \\
\hline IX& ${\mathcal Z}$ & ${\mathcal Z}$ & $0$ \\
\hline \hline
\end{tabular}
\end{center}
where ${\mathcal L}(A)$ denotes the Lebesgue measure of the set $A$:
\[
\mathcal{L}(A) := \int_{A} 1 d\eta d\Omega.
\]
We also note that
\[
\mathcal{L}(\mathcal{P})+\mathcal{L}(\mathcal{Z})+\mathcal{L}(\mathcal{N}) = \int_{\bbr} \int_{0}^{g(\Omega)} 1 d\eta d\Omega = \int_{\bbr} g(\Omega) d\Omega = 1.
\]
\noindent {\bf Case I}: In this case, we use the definition of $\Delta(\eta, \eta_*,\Omega,\Omega_*)$ and the inequality
\[ \sin x \geq \frac{2}{\pi}x, \qquad \mbox{for}~~ x \in \Big [0, \frac{\pi}{2} \Big], \]
to determine that
\[ \Delta(\eta, \eta_*,\Omega,\Omega_*) = -|\Phi_*|^{p-1}\sin \frac{|\Phi_*|}{2} \leq  -\frac{1}{\pi} |\Phi_*|^p.  \]
This yields
\[ I({\mathcal P}, {\mathcal Z}) \leq -\frac{1}{\pi} \int_{{\mathcal P} \times {\mathcal Z}} |\Phi_*|^p d\eta d\eta_* d\Omega d\Omega_*=
-\frac{{\mathcal L}({\mathcal Z})}
{\pi}  \int_{{\mathcal P}}|\Phi_*|^p d\eta_* d\Omega_*. \]

\noindent {\bf Case II} - \noindent {\bf Case IV}: The estimates are basically the same as in Case I. Hence, we omit their estimates. \newline

\noindent {\bf Case V}: In this case, we have
\begin{eqnarray*}
 \Delta(\eta, \eta_*,\Omega,\Omega_*) &=& -(|\Phi|^{p-1}+|\Phi_*|^{p-1})\sin \left(\frac{|\Phi_*| + |\Phi|}{2} \right)   \cr
                             &\leq& -\frac{1}{\pi} \Big( |\Phi|^{p}+|\Phi_*|^{p} + |\Phi|^{p-1}|\Phi_*| +
                             |\Phi_*|^{p-1}|\Phi|  \Big).
\end{eqnarray*}
This yields the desired result.

\noindent {\bf Case VI}: Once we interchange ${\mathcal P} \longleftrightarrow {\mathcal N}$, the same estimate holds. \newline

\noindent {\bf Case VII}: In this case, we need to consider two subcases:
\[ \mbox{Either}~~\Phi > \Phi_* > 0 \quad \mbox{or} \quad \Phi_* \geq \Phi > 0. \]
By considering each case, we have
\begin{eqnarray*}
 \Delta(\eta, \eta_*,\Omega,\Omega_*) &=& \left(|\Phi|^{p-1}-|\Phi_*|^{p-1}\right) \sin \Big(\frac{\Phi_*-\Phi}{2}\Big) \cr
                             &\leq& \frac{1}{\pi} \left(|\Phi|^{p-1}-|\Phi_*|^{p-1}\right)\left( |\Phi_*|-|\Phi|\right) \cr
                             &=&-\frac{1}{\pi} \Big( |\Phi|^{p}+|\Phi_*|^{p} - |\Phi|^{p-1}|\Phi_*| -
                             |\Phi_*|^{p-1}|\Phi|  \Big).
\end{eqnarray*}
This yields the desired result. \newline

\noindent {\bf Case VIII}:: The estimate is exactly the same as in Case VII. Hence we omit its estimate. \newline

\noindent {\bf Case IX}:: The estimate is trivial. \newline

We now add all cases and use \eqref{zero} to find
\begin{align*}
\begin{aligned}
&\int_{\bbr} \int_{\bbr} \int_0^{g(\Omega_*)}\int_0^{g(\Omega)} \Delta(\eta, \eta_*,\Omega,\Omega_*) d\eta d\eta_* d\Omega d\Omega_* \cr
& \qquad \qquad \qquad \leq -\frac{2}{\pi} \Big( {\mathcal L}({\mathcal P}) + {\mathcal L}({\mathcal Z}) + {\mathcal L}({\mathcal N})
       \Big) \int_{\bbr} \int_0^{g(\Omega)} |\Phi|^p d\eta d\Omega \cr
& \qquad \qquad \qquad = -\frac{2}{\pi} \int_{\bbr} \int_0^{g(\Omega)} |\Phi|^p d\eta d\Omega.
\end{aligned}
\end{align*}
\end{proof}

\begin{theorem}\label{thm-wp}
Suppose that two initial measures $\mu_0,~\nu_0 \in {\mathcal M}([0, 2\pi) \times \bbr)$ and $K$ satisfy
\begin{eqnarray*}
&& (i)~0<D_{\theta}({\nu_0}) \le D_{\theta}({\mu_0}) <\pi,  \quad \int_{[0,2\pi] \times \bbr} \theta \mu_0( d\theta, d\Omega) = \int_{[0,2\pi] \times \bbr}
\theta \nu_0( d\theta, d\Omega) =\pi. \cr
&& (ii)~K > D_{\Omega}(\mu_0) \max \Big \{ \frac{1}{\sin D_{\theta}(\mu_0)},
\frac{1}{\sin D_{\theta}(\nu_0)} \Big \},
\end{eqnarray*}
and let $\mu_t$ and $\nu_t$ be two measure valued solutions to
\eqref{k-ku} - \eqref{ini} corresponding to initial data $\mu_0$
and $\nu_0$, respectively. Then, there exists $t_0 > 0$ such that
\[\widetilde{W}_p(\mu_t,\nu_t) \leq \exp \Big[ -\frac{2K \cos D^\infty}{\pi} (t-t_0) \Big ] \widetilde{W}_p(\mu_{t_0},\nu_{t_0}),\quad t>t_0,~~1 \leq p \leq \infty. \]
\end{theorem}

\begin{proof}
First, we consider the case where $p \in [1, \infty)$. Note that
the Wasserstein distance in one-space dimension is equivalent to
the $L^p$-distance of its corresponding pseudo inverse
distribution function. Remember that we are assuming that the
solutions are smooth, hence it is more convenient to obtain the
$L^p$-estimate from equation \eqref{phi-t}. Denoting by
$\phi_i,~i=1,2$ the pseudo inverse functions associated to $\mu_t$
and $\nu_t$ respectively, we get
\begin{align*}
\partial_t \phi_i &= \Omega + K \int_{\bbr} \int_0^{g(\Omega_*)}  \sin (\phi_{i*} -\phi_i)
d\eta_* d\Omega_* \,,
\end{align*}
for $i=1,2$. Then the above equations imply that
\begin{align}
\begin{aligned} \label{EQ-2}
& \partial_t (\phi_1 - \phi_2) = K \int_{\bbr}
\int_0^{g(\Omega_*)}  \Big( \sin (\phi_{1*} -\phi_1)  - \sin
(\phi_{2*}-\phi_2) \Big)
d\eta_* d\Omega_*, \\
                             & \hspace{0.5cm} = 2K \int_{\bbr} \int_0^{g(\Omega_*)}  \cos\left(\frac{\phi_{1*}-\phi_1}{2}+\frac{\phi_{2*} -\phi_{2}}{2} \right) \sin \left(\frac{\phi_{1*} -\phi_1}{2}-\frac{\phi_{2*}-\phi_2}{2}
                             \right) d\eta_* d\Omega_*.
\end{aligned}
\end{align}
We multiply \eqref{EQ-2} by $p\mbox{sgn}(\phi_1 - \phi_2) |\phi_1 -\phi_2|^{p-1}$ and integrate over $[0, g(\Omega)] \times \bbr$ using the symmetry $(\eta,\Omega) \Longleftrightarrow (\eta_*,\Omega_*)$
to obtain
\begin{eqnarray*}
&& \frac{d}{dt} ||\phi_1 - \phi_2||_{L^p}^p \cr
&& \hspace{0.5cm}= 2pK\int_{\bbr} \int_{\bbr} \int_0^{g(\Omega)} \int_0^{g(\Omega_*)}  \Big[ \cos\left(\frac{\phi_{1*}-\phi_{1}}{2}+\frac{\phi_{2*}-\phi_{2}}{2} \right)
\sin \left(\frac{\phi_{1*}-\phi_{1}}{2}-\frac{\phi_{2*}-\phi_{2}}{2} \right) \cr
&& \hspace{2cm} \times \left[|\phi_{1}-\phi_{2}|^{p-1}\mbox{sgn} \left(\phi_{1}-\phi_{2} \right) -|\phi_{1*}-\phi_{2*}|^{p-1}\mbox{sgn} \left(\phi_{1*}-\phi_{2*} \right)\right] \Big] d\eta_* d\eta d\Omega_* d\Omega.
\end{eqnarray*}
It follows from the proof of Lemma \ref{lemma-cal} that for all
$a,~b \in \mathbb{R}$,
\begin{equation*}
\big( |a|^{p-1}\mbox{sgn} (a) - |b|^{p-1}\mbox{sgn}(b) \big) \sin \left(\frac{b -a}{2}\right) \leq 0.
\end{equation*}
On the other hand, Lemma \ref{lemma-appli-1} implies that there
exists $t_0$ such that
\[
D_{\theta}({\mu_t}) \le D^{\infty},\quad D_{\theta}({\nu_t}) \le D^{\infty}, \quad t \geq t_0,
\]
and we use Lemma \ref{lemma-appli-2} to obtain
\begin{eqnarray*}
&&\max_{\Omega\in \mbox{spt}(g)}\phi_1(g(\Omega),\Omega,t)-\min_{\Omega\in
\mbox{spt}(g)}\phi_1(0,\Omega,t)\le D^{\infty},\\
&&\max_{\Omega\in \mbox{spt}(g)}\phi_2(g(\Omega),\Omega,t)-\min_{\Omega\in
\mbox{spt}(g)}\phi_2
(0,\Omega,t)\le D^{\infty} ,\quad t\ge t_0.
\end{eqnarray*}
Then, this yields
\[
0 < \cos D^{\infty} \leq  \cos \left( \frac{\phi_{1*}-
\phi_{1}}{2}+\frac{\phi_{2*}- \phi_{2}}{2}
\right) . \] Hence, we obtain
\begin{equation*}
\frac{d}{dt} ||\phi_1 - \phi_2||_{L^p}^p \leq 2pK \cos D^\infty \mathcal{J}
\end{equation*}
where
\begin{eqnarray*}
&& \mathcal{J}:= \int_{\bbr \times \bbr} \int_0^{g(\Omega)} \int_0^{g(\Omega_*)} \sin \left(\frac{\phi_{1*}-\phi_{2*}}{2}-\frac{\phi_{1}-\phi_{2}}{2} \right)\\
&& \hspace{1cm} \times \left[|\phi_{1}-\phi_{2}|^{p-1}\mbox{sgn} \left(\phi_{1}-\phi_{2} \right)-|\phi_{1*}-\phi_{2*}|^{p-1}\mbox{sgn} \left(\phi_{1*}-\phi_{2*} \right)\right] d\eta_* d\eta d\Omega_* d\Omega.
\end{eqnarray*}
If we set $\Phi := \phi_{1}-\phi_{2}$, then
\[
|\Phi_* - \Phi|\le
|\phi_{1*}-\phi_{1}| +
|\phi_{2*}-\phi_{2}| \le 2 D^{\infty} < \pi,\quad
t>t_0.
\]
Since $\mu_0$, $\nu_0$ have the same center of mass, it follows
from Lemma \ref{contract-lemma1} that
\[
\int_{\bbr} \int_0^{g(\Omega)} \Phi d\eta d\Omega= 0,\quad t>t_0.
\]
Thus, we can apply Lemma \ref{lemma-cal} with
$\Phi = \phi_{1}-\phi_{2}$ to obtain
\[
\frac{d}{dt} \widetilde{W}_p^p(\mu_t,\nu_t) \leq -\frac{2pK\cos
D^\infty }{\pi} \widetilde{W}_p^p(\mu_t,\nu_t), \quad t \geq t_0.
\]
This yields
\begin{equation}\label{lp-stability-est}
\widetilde{W}_p(\mu_t,\nu_t) \leq \exp \Big( -\frac{2K \cos
D^\infty}{\pi} (t - t_0) \Big)
\widetilde{W}_p(\mu_{t_0},\nu_{t_0}).
\end{equation} In the case of $p=\infty$, we use \eqref{limit} and
\eqref{lp-stability-est} to obtain
\[
\widetilde{W}_{\infty}(\mu_t,\nu_t) \leq \exp \Big( -\frac{2K \cos
D^\infty}{\pi} (t - t_0) \Big)
\widetilde{W}_{\infty}(\mu_{t_0},\nu_{t_0}).
\]
This completes the proof for smooth solutions. As mentioned above
a simple approximation argument as in Subsection 3.3 finishes the
proof for measure valued solutions.
\end{proof}

\begin{remark}
The assumption in Theorem \ref{thm-wp} on the initial measures to
have equal mean in $\theta$ is not restricted. Due to Lemma
\ref{contract-lemma1}, the mean in $\theta$ is preserved in time.
Thus, we can always restrict to the equal mean in $\theta$ case by
translational invariance of \eqref{k-ku}.
\end{remark}

%

\end{document}